\documentclass[fleqn,11pt, reqno]{amsart}
\usepackage{indentfirst, amsfonts, amsmath, amsthm, amssymb, amscd}
\usepackage{amsmath,amsfonts,amscd,bezier}
\usepackage[all]{xy}
\usepackage{hyperref}
\usepackage[utf8]{inputenc}
\usepackage{psfrag}
\usepackage[dvips]{graphicx}
\usepackage{tikz} 

\usepackage{amsmath,amsfonts,latexsym}
\usepackage{amscd,amssymb,amsmath}
\usepackage{amsfonts}
\usepackage{amsbsy}
\usepackage{epsfig,afterpage}
\usepackage{psfrag}
\usepackage{graphicx}
\usepackage{colortbl}
\usepackage{cancel}
\usepackage[all]{xy}
\usepackage{hyperref}

\newtheorem{lema}{Lemma}[section]
\newtheorem{teo}[lema]{Theorem}
\newtheorem{propo}[lema]{Proposition}

\newtheorem{coro}[lema]{Corollary}
\newtheorem{defi}[lema]{Definition}

\newcommand{\bea}{\begin{eqnarray*}}
\newcommand{\eea}{\end{eqnarray*}}
\newcommand{\zz}[1]{}

\begin{document}

\thanks{$^{2}$Supported partially by FAPESP of Brazil, Grant 2011/22285-1, University of São Paulo, Brazil.}

\title{Homotopy of Braids on Surfaces: extending Goldsmith's Answer to Artin}
\author[Juliana Roberta Theodoro de Lima]{Juliana Roberta Theodoro de Lima $^{1}$ }
\address{Instituto de Matemática, Universidade Federal de Alagoas,\\ 
Campus A. C. Simões, Maceio, Lourival Melo Motta Avenue, no number\\
CEP: 57072-970
}
\email[Juliana \ R. \ Theodoro \ de \ Lima]{juliana.lima@im.ufal.br} 

 \maketitle

\begin{abstract}

In 1947, in the paper "Theory of Braids", Artin raised the question of whether isotopy and homotopy of braids on the disk coincide. Twenty seven years later, Goldsmith answered his question: she proved that in fact the group structures are different, exhibiting a group presentation and showing that the homotopy braid group on the disk is a proper quotient of the Artin braid group on the disk $B_{n}$, denoted by $\widehat{B}_{n}$. In this paper, we extend Goldsmith's answer to Artin for closed, connected and orientable surfaces different from the sphere. More specifically, we define the notion of homotopy generalized string links on surfaces, which form a group which is a proper quotient of the braid group on a surface $B_{n}(M)$, denoting it by $\widehat{{B}}_{n}(M)$. We then give a presentation of the group $\widehat{B}_{n}(M)$ and find that the Goldsmith presentation is a particular case of our main result, when we consider the surface $M$ to be the disk. We close with a brief discussion surrounding the importance of having such a fixed construction available in the literature.
\end{abstract}

\section{\bf{Introduction}}

\vspace{0.5cm}

In $1947$, E. Artin introduced the study of braids with his pioneering paper called \emph{Theory of Braids} (see \cite{Artin}), which is directly related to knots and links theory. Although braids, links and knots had already been discussed earlier, Artin showed two important results for the theory: the presentation and representation theorems for the braid group on the disk, namely $B_{n}$, also known as the Artin braid group. For our purposes here, we focus on the first result: a presentation of a group is a way to represent a group by generators and relations. The braid group is a group of equivalence classes, where the equivalence relation is isotopy (or, more formally, ambient isotopy). However, in the same paper Artin proposed the idea of homotopy braids: essentially, it is the same set divided into equivalence classes using the equivalence relation of homotopy. The operation (concatenation) remains the same among braids. Accordingly, he posed the following questions: would the homotopy braids on the disk have the same properties, group structure and presentation as braid groups? Otherwise, what are its differences?

Goldsmith \cite{Goldsmith} answered all these questions: in fact, she proved that the group structures are different, making it explicit when certain types of braids are not trivial up to isotopy but trivial up to homotopy. Furthermore, she provided a presentation for homotopy braid groups on the disk, denoted by $\widehat{B}_{n}$. Homotopy has been discussed since the beginning of the formalization of the studies of braid groups presented by Artin. However, homotopy braid theory was formalized by Milnor some years after Artin's seminal paper in \cite{Milnor} and it has been extended with the works \cite{Habegger, Levine, Yurasovskaya}. Moreover, there is still a slight difference between the concepts of string links given in \cite{Milnor, Yurasovskaya} and of homotopy braids given in \cite{Goldsmith}: string links are pure braids, either on the disk or on surfaces, with the monotonicity requirement relaxed, whereas homotopy braids are braids on the disk (not necessarily pure) with the monotonicity requirement relaxed. Consequentely, we see that the most recent works are restricted to the pure case and, therefore, it is reasonable to inquire about the general case. 

In order to keep our notation and terminology in line with the recent literature, we use the terms \emph{string links} instead of homotopy braids and up to \emph{link-homotopy} instead of homotopy. 

Among the significant number of interactions of braid and homotopy braid theory with other areas of low-dimensional topology and algebra, we would like to highlight the theory of orderable groups: we say that a group $G$ is left-orderable if there is a strict linear ordering $<$ of G such that $g<h$ implies $fg < fh$, for all $f,g, h \in G$. Similarly, we say that $G$ is bi-orderable if there is a strict linear ordering $<$ of G such that $g<h$ implies $fgs < fhs$, for all $f,g, h, s \in G$. Whether or not a group or its subgroups are left-orderable can carry important algebraic information about the group, such as torsion-freeness, absence of zero divisors in their group algebra (more specifically in $RB_{n}$, where $R$ is a ring without zero divisors), a solution for the word problem, faithfulness criteria for representations of braid groups and recent studies on detecting prime knots and links, which can be seen in more details in \cite{Rolfsen} and \cite{Rolfsen2}.

In 1994, Dehornoy described a left-ordering for $B_{n}$ \cite{Dehornoy}.  In $1999$, Rolfsen and Zhu described a bi-ordering for the pure braid group on the disk, namely $PB_{n}$ \cite{Rolfsen2}. Later, in $2002$, Gonzalez-Meneses proved that pure braid group on a closed, connected and orientable surface $M$ of genus $g \geq 1$, namely $PB_{n}(M)$, is bi-orderable \cite{Gonzales2}. On the other hand, it is not known if the braid groups on closed surfaces are left-orderable or not and it has been a much-discussed topic in the literature \cite{Rolfsen, Gonzales2}. In $2008$, Yurasovskaya  \cite{Yurasovskaya} proved that group of homotopy string links on the disk, namely $\widehat{PB}_{n}$, is bi-orderable. This last result was extended by Lima and Mattos \cite{Lima} for homotopy string links on closed, connected and orientable surfaces $M$ of genus $g \geq 1$, namely $\widehat{PB}_{n}(M)$. It is 
worth noting that all studies mentioned in this paragraph were successful in part due to availability of a finite presentation for the groups under consideration. Note that all the groups mentioned above have an orientable underlying surface (with or without boundary and different from the sphere): that is why we will focus on these types of surfaces in this paper.

This paper is organized as follows: in Section \ref{secao2} we state all the results that will be useful for our constructions. Section \ref{secao3} is divided into two subsections: firstly, we extend the definition of string links on the disk, string links on surfaces (pure) given in \cite{Habegger, Levine, Milnor, Yurasovskaya} for the general case (not pure), called {\it{generalized string links on surfaces}} (not necessarily orientable). Generalized string links on surfaces extend the definitions of string links and of homotopy braids, covering the need of being explicit to avoid misunderstandings. Furthermore, we extend the definitions of link-homotopy equivalence relation given in \cite{Habegger, Levine} for the general case (not pure) and provide a group structure (with concatenation operation) up to link homotopy called {\it{the group of homotopy generalized string links over surfaces}}, namely $\widehat{B}_{n}{(M)}$. We also show that $\widehat{PB}_{n}(M)$ given in \cite{Yurasovskaya} is a normal subgroup of $\widehat{B}_{n}{(M)}$ and we prove that $\widehat{B}_{n}{(M)}$ is a  proper quotient of $B_{n}(M)$. Secondly, we extend Goldsmith's paper by providing a presentation for homotopy generalized string links over surfaces. Goldsmith's presentation is a particular case of our presentation, when we consider the disk instead of surfaces. The study described in this paper will not only be a motivation, but also a tool to move forward in search for information about orderability of $B_{n}(M)$ and $\widehat{B}_{n}{(M)}$. We present a brief comment about this at the end of Section \ref{secao3}.

\vspace{2cm}

\section{\bf{Statements}\label{secao2}}

\subsection{\bf{Braids and string links over surfaces}\label{resumo2}}

\begin{defi}\label{definicaoproblema}\rm{\cite[p.431]{Gonzales}}
Let $M$ be a closed surface, not necessarily orientable, and let $\mathcal{P}= \{P_{1} , \ldots ,P_{n}\}$ be a set of $n$ distinct points of $M$. A geometric braid over $M$ based at $\mathcal{P}$ is an $n$-tuple $\gamma= ( \gamma_{1},\ldots,\gamma_{n})$ of paths, $\gamma_{i}: [0,1] \rightarrow M$, such that:
\begin{itemize}
\item[(1)] $\gamma_{i}(0) = P_{i}, \hbox{for all} \ i= 1,\ldots,n$,
\item[(2)] $\gamma_{i}(1) \in \mathcal{P}, \ \hbox{for all} \ i = 1,\ldots,n$,
\item[(3)] $\{\gamma_{1}(t),\ldots,\gamma_{n}(t)\} \;\hbox{are $n$ distinct points of $M$, for all} \ t \in [0,1]$.
\end{itemize}
For all $i= 1,\ldots, n$, we will call $\gamma_{i}$ the $i$-th strands (or strings) of $\gamma$.
\end{defi}

We say two geometric braids $\beta$ and $\alpha$ are {\it{isotopic}} if there exists an ambient isotopy which deforms one to the other, with endpoints fixed during the deformation process. The set of all equivalence classes of geometric braids on $n$-strands on the surface $M$ forms a group called the \emph{braid group on $n$ strings on a surface $M$}, namely $B_{n}(M)$, equipped with the operation (product) called concatenation. The inverse of each braid $\gamma$ is given by the mirror reflection of $\gamma$. If the surface is the disk $\mathbb{D}$, then $B_{n}(\mathbb{D})$ is the Artin braid group $B_{n}$.

\begin{figure}[h]
\center
\includegraphics[scale=0.3]{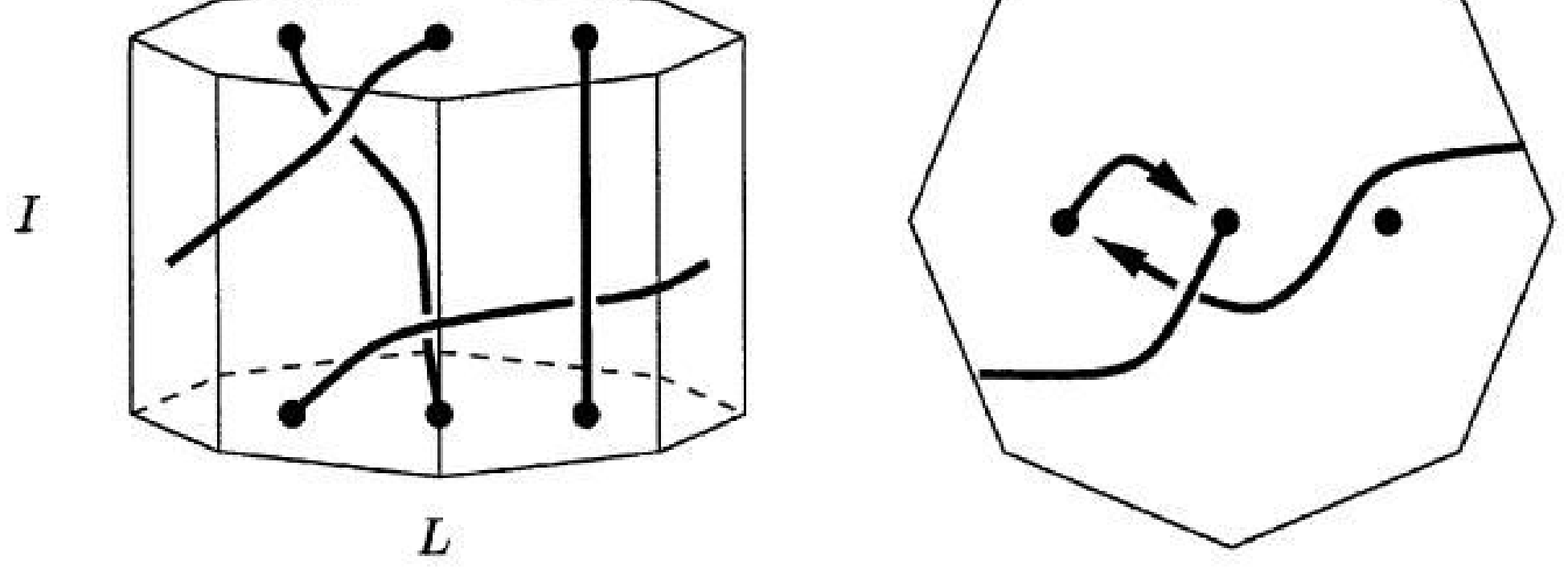}
\vspace{-6.2cm}
\caption{Different viewpoints of a braid on a surface \cite{Gonzales}.\label{FigGonza1}}
\end{figure}

\begin{teo}{\rm{\cite[Theorem 2.1]{Gonzales}}}\label{Teorema1.1.2}
If $M$ is a closed, orientable surface of genus $g \geq 1$, then $B_{n}(M)$ admits the following presentation:
\begin{itemize}
\item[] {\rm{\bf{Generators:}}} $\sigma_{1},\ldots, \sigma_{n-1}, \mathit{a}_{1,1},\ldots,\mathit{a}_{1,2g}$.
\item[] {\rm{\bf{Relations:}}}
\begin{itemize}
\item[{\rm(R1)}] $\sigma_{i}\sigma_{j}= \sigma_{j}\sigma_{i} \hspace{10.1cm} |i-j|\geq 2$;
\item[{\rm(R2)}] $\sigma_{i}\sigma_{i+1}\sigma_{i} = \sigma_{i+1}\sigma_{i}\sigma_{i+1} \hspace{7.8cm} 1\leq i \leq n-2$;
\item[{\rm(R3)}] $\mathit{a}_{1,1}\cdots\mathit{a}_{1,2g}\mathit{a}_{1,1}^{-1}\cdots\mathit{a}_{1,2g}^{-1}=\sigma_{1}\cdots\sigma_{n-2}\sigma_{n-1}^{2}\sigma_{n-2}\cdots\sigma_{1}$
\item[{\rm(R4)}] $\mathit{a}_{1,r}A_{2, s}= A_{2, s}\mathit{a}_{1,r} \hspace{5.0cm} 1\leq r\leq 2g; \; 1\leq s\leq 2g-1;  \ r\neq s$;
\item[{\rm(R5)}] $(\mathit{a}_{1,1}\cdots\mathit{a}_{1,r})A_{2, r}=
\sigma_{1}^{2}A_{2, r}(\mathit{a}_{1,1}\cdots\mathit{a}_{1,r}) \hspace{4.9cm} 1\leq r \leq 2g-1$;
\item[{\rm(R6)}] $\mathit{a}_{1,r}\sigma_{i}= \sigma_{i}\mathit{a}_{1,r} \hspace{8.5cm} 1\leq r \leq 2g; \  i\geq 2$.
\end{itemize}
where:\\  $$A_{2,r}= \sigma_{1}^{-1}(a_{1,1}\cdots a_{1,r-1}a_{1,r+1}^{-1}\cdots a_{1,2g}^{-1})\sigma_{1}^{-1}.$$
\end{itemize}
\end{teo}

\vspace{0.5cm}

\begin{figure}[h]
\center
\includegraphics[scale=0.4]{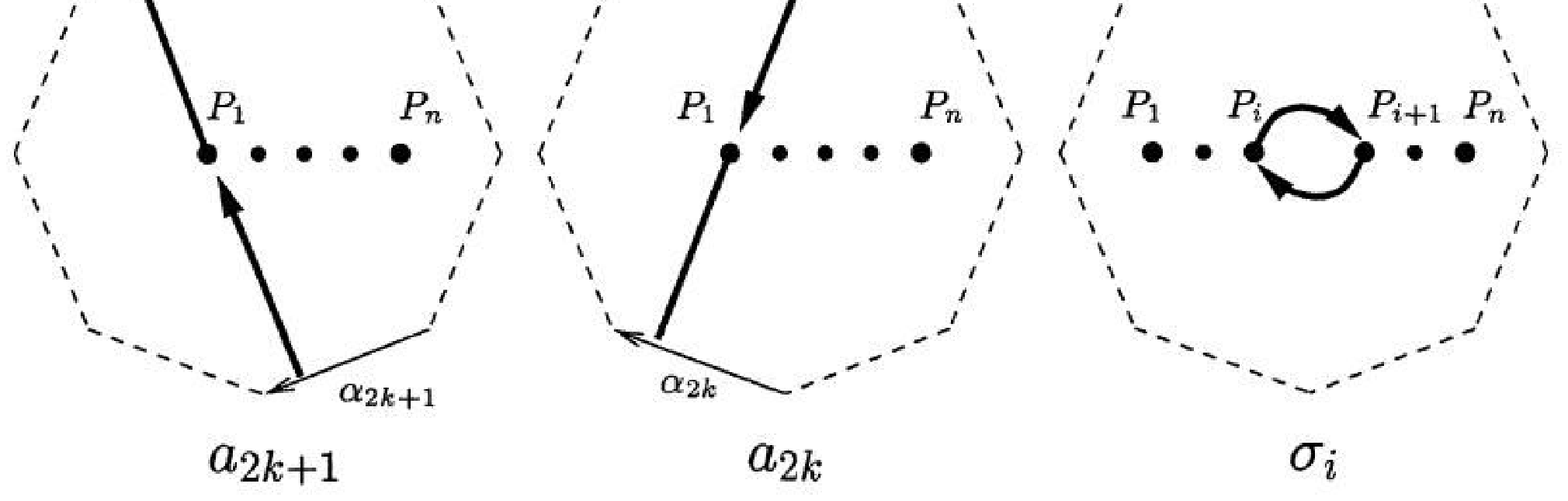}
\vspace{-8.8cm}
\caption{The generators of $B_{n}(M)$, where $a_{r}=a_{1,r}$ \cite{Gonzales}.\label{Fig1.3}}
\end{figure}

\vspace{2cm}

\begin{defi}\rm{\cite[p.12]{Yurasovskaya}}\label{def_stringlink}
Let $M$ be a compact orientable surface of genus $g \geq 1$. Choose $n$ points $\mathcal{P}= \{P_{1},\ldots, P_{n}\}$ to lie in the interior of $M$. Let $I_{1},\ldots,I_{n}$ be $n$ copies of the  unit interval $I=[0,1]$ and $\coprod_{i=1}^{n}{I_{i}}$ denote the disjoint union of these intervals. A string link $\sigma$ on $n$-strands on the surface $M$ is a smooth or piecewise linear proper embedding $\sigma: \coprod_{i=1}^{n}{I_{i}} \rightarrow M \times I$ such that $\sigma_{|_{I_{i}(0)}}= (P_{i},0)$ and $\sigma_{|_{I_{i}(1)}}= (P_{i},1)$.
\end{defi}

Observe that a string link is pure, i.e., it induces the trivial permutation on its strands.

\begin{figure}[h]
\center
\includegraphics[scale=0.25]{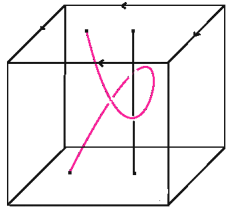}
\vspace{-0.3cm}
\caption{A string link $\sigma$ on the $2$-dimensional torus. \label{Fig2.1}}
\end{figure}

\begin{defi}\cite{Habegger}\label{linkhomotopiaH}
We say that two string links $\sigma$ and $\sigma^{\prime}$ are \emph{link-homotopic} if there is a homotopy of the strands in $M \times I$, fixing $M \times \{0,1\}$ and deforming $\sigma$ to $\sigma^{\prime}$, such that the images of different strands remain disjoint during the deformation.
\end{defi}

During the deformation, each individual strand is allowed to pass through itself but not through the others.

We say that link-homotopy is an equivalence relation on string links, which is generated by a sequence of ambient isotopies of $M \times I$ fixing $M \times \{0,1\}$, and local {\it{crossing changes}} of arcs from the same strand of a string link. This alternative definition of link-homotopy can be found in \cite{Milnor, Levine, Habegger} and it gives us a geometric intuition about link-homotopy deformation. 

\begin{figure}[h]
\center
\includegraphics[scale=1.5]{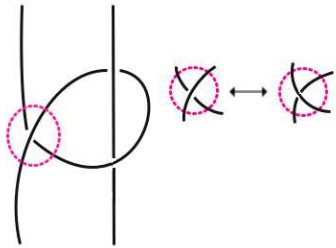}
\vspace{-0.2cm}
\caption{A crossing change of an arc from the same strand. \label{Fig2.1.1}}
\end{figure}

The set of all equivalence classes of string links on $n$-strands on the surface $M$ forms a group called the \emph{group of homotopy string links on $n$ strings on a surface $M$}, denoted by $\widehat{PB}_{n}(M)$. Again, the operation is concatenation and the inverse of each string link $\gamma$ is given by the mirror reflection of $\gamma$. 

Rolfsen and Fenn proved that every string link is link-homotopic to a pure braid (a proof can be found in \cite[Theorem $3.7$]{Yurasovskaya}). Also, Yurasovskaya proved in \cite[Lemma $3.8$, Proposition $3.9$]{Yurasovskaya} that the link-homotopically trivial braids $H_{n}(M)$ form a normal subgroup of $PB_{n}(M)$ and $\widehat{PB}_{n}(M)$ is a proper quotient of $PB_{n}(M)$ by $H_{(n)}(M)$, i.e., $\widehat{PB}_{n}(M) \simeq {PB_{n}(M)}/{H_{n}(M)}$. Furthermore,  $H_{n}(M)$ can be seen as the smallest normal subgroup of $PB_{n}{(M)}$ generated by $[t_{i,j},t_{i,j}^{h}]$, where 

\begin{eqnarray*}
t_{i,j} &= & \sigma_{i}\sigma_{i+1}\cdots\sigma_{j-2}\sigma_{j-1}^{2}\sigma_{j-2}^{-1}\cdots\sigma_{i+1}^{-1}\sigma_{i}^{-1},\\
t_{i,j}^{h} &=& ht_{i,j}h^{-1}, h \in \mathbb{F}(2g + n- i),
\end{eqnarray*} and $\mathbb{F}(2g + n- i)$ is the free group on $2g+n-i$ generators  $\pi_{1}(M\setminus \mathcal{P}_{n-i}; P_{i})$ generated by $\{\{a_{i,r}\} \cup \{t_{i,j}\}; i+1 \leq j \leq n, 1 \leq r \leq 2g \}$, with $\mathcal{P}_{n-i}= \{P_{i+1}, \ldots, P_{n}\}$ and $M$ is a closed, connected and orientable surface of genus $g \geq 1$ (see \cite[Proposition $4.5$]{Yurasovskaya}). In symbols: $$H_{n}(M)= <\{[t_{i,j},t_{i,j}^{h}]; \ 1 \leq i < j \leq n, h \in \mathbb{F}(2g + n- i)\}>^{N},$$ where $<>^{N}$ denotes the normal closure. The representations of $a_{i,r}$ and $t_{i,j}$ are given in Figure $\ref{disney3}$.

\begin{figure}[h]
\center
\includegraphics[scale=0.5]{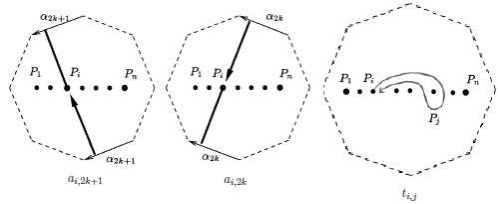}
\vspace{-0.3cm}
\caption{The braids $a_{i,r}$ and $t_{i,j}$ \cite{Yurasovskaya}.\label{disney3}}
\end{figure} 


A representation of a relation of $H_{n}(M)$ is given in Figure $\ref{Fig2.3}$.\\

\begin{figure}[h]
\center
\includegraphics[scale=0.55]{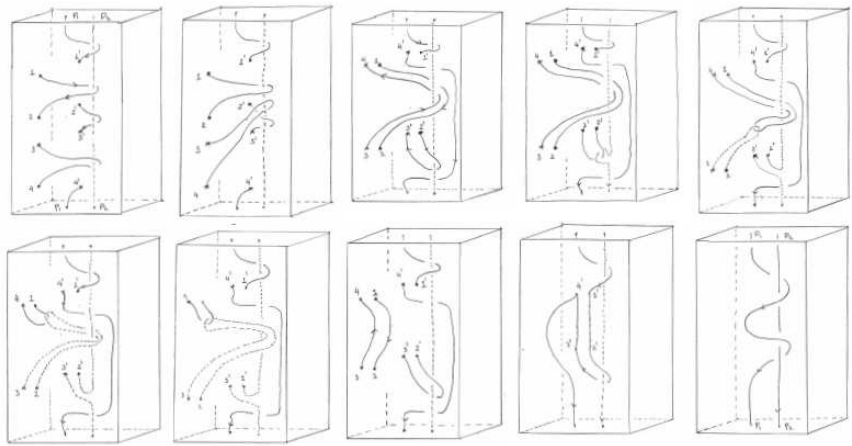}
\vspace{-0.2cm}
\caption{Particular case of $[t_{i,j}, t_{i,j}^{h}]$ up to link-homotopy \cite{Yurasovskaya}.\label{Fig2.3}}
\end{figure} 


\vspace{7.5cm}
Under the notations above, we have the following:

\begin{teo}{\rm{\cite[Theorem 6.3]{Yurasovskaya}}}\label{Teorema2.1.5}
Let $M$ be a closed, compact, connected and orientable surface of genus $g \geq 1$. The group of homotopy string links $\widehat{PB}_{n}(M)$ admits the following presentation:
\begin{itemize}
\item[]{\bf{Generators:}} $\{a_{i,r}; \;\;1\leq i \leq n; \;\; 1 \leq r \leq 2g\}\cup \{t_{j,k};\;\; 1 \leq j < k \leq n\}.$
\item[]{\bf{Relations:}}
\begin{itemize}\small{
\item[{\rm{(LH1)}}] $[t_{i,j},t_{i,j}^{h}]= 1$ \hspace{9.1cm} $h \in \mathbb{F}(2g + n- i)$;
\item[{\rm(PR1)}] $a_{n, 1}^{-1}a_{n, 2}^{-1}\cdots a_{n, 2g}^{-1}a_{n, 1}a_{n, 2}\cdots a_{n, 2g}= \displaystyle\prod_{i= 1}^{n-1}{T_{i, n-1}^{-1}T_{i, n}}$;
\item[{\rm(PR2)}] $a_{i, r}A_{j, s}= A_{j, s}a_{i, r} \hspace{3.6cm} 1\leq i<j \leq n, \; 1 \leq r \leq 2g;\; 1 \leq s \leq 2g-1; \;\; r\neq s$;
\item[{\rm(PR3)}] $(a_{i, 1}\cdots a_{i, r})A_{j, r}(a_{i, r}^{-1}\cdots a_{i, 1}^{-1})A_{j, r}^{-1}= T_{i, j}T_{i, j-1}^{-1}$ $\hspace{2.2cm} 1\leq i<j \leq n, \;\; 1\leq r \leq 2g-1$;
\item[{\rm(PR4)}] $T_{i, j}T_{k, l}= T_{k, l}T_{i, j} \hspace{3.9cm} 1\leq i<j<k<l \leq n \;\; \text{or} \;\; 1\leq i<k<l\leq j\leq n$;
\item[{\rm(PR5)}] $T_{k, l}T_{i, j}T_{k, l}^{-1}= T_{i, k-1}T_{i, k}^{-1}T_{i, j}T_{i, l}^{-1}T_{i, k}T_{i, k-1}^{-1}T_{i, l}$ $\hspace{3.6cm} 1\leq i<k \leq j<l \leq n$;
\item[{\rm(PR6)}] $a_{i, r}T_{j, k}= T_{j, k}a_{i, r}  \hspace{3cm} 1\leq i<j<k \leq n \;\; \text{or} \;\; 1\leq j<k<i \leq n, \;\; 1\leq r \leq 2g $;
\item[{\rm(PR7)}] $a_{i, r}(a_{j, 2g}^{-1}\cdots a_{j, 1}^{-1}T_{j, k}a_{j, 2g}\cdots a_{j, 1}) = (a_{j, 2g}^{-1}\cdots a_{j, 1}^{-1}T_{j, k}a_{j, 2g}\cdots a_{j, 1})a_{i, r} \hspace{0.8cm} 1\leq j<i \leq k \leq n$;
\item[{\rm(PR8)}] $T_{j,n}= \left(\displaystyle\prod_{i= 1}^{j-1}{a_{i, 2g}^{-1}\cdots a_{i, 1}^{-1}T_{i, j-1}T_{i, j}^{-1}a_{i, 1}\cdots a_{i, 2g}}\right)a_{j, 1}\cdots a_{j, 2g}a_{j, 1}^{-1}\cdots a_{j, 2g}^{-1}$;}
\end{itemize}
where:\\ 
\begin{align*}
A_{j, s} &= a_{j, 1}\cdots a_{j, s-1}a_{j, s+1}^{-1}\cdots a_{j, 2g}^{-1},\\
T_{i,j}   &=  t_{i,j}\cdots t_{i,i+1}.
\end{align*}
\end{itemize}
\end{teo}

\vspace{0.5cm}

\begin{teo}{\rm{\cite[Theorem 5.1]{Yurasovskaya}}}\label{Teorema2.6}
Let $S$ be the surface obtained by deleting a single point from a closed, compact, connected and orientable surface of genus $g \geq 1$. The group of homotopy string links $\widehat{PB}_{n}(S)$ admits the following presentation:
\begin{itemize}
\item[]{\bf{Generators:}} $\{a_{i,r}; \;\;1\leq i \leq n; \;\; 1 \leq r \leq 2g\}\cup \{t_{j,k};\;\; 1 \leq j < k \leq n\}.$
\item[]{\bf{Relations:}}
\begin{itemize}\small{
\item[{\rm{(LH1)}}] $[t_{i,j},t_{i,j}^{h}]= 1$ \hspace{9.1cm} $h \in \mathbb{F}(2g + n- i)$;
\item[{\rm(PR2)}] $a_{i, r}A_{j, s}= A_{j, s}a_{i, r} \hspace{3.6cm} 1\leq i<j \leq n, \; 1 \leq r \leq 2g;\; 1 \leq s \leq 2g-1; \;\; r\neq s$;
\item[{\rm(PR3)}] $(a_{i, 1}\cdots a_{i, r})A_{j, r}(a_{i, r}^{-1}\cdots a_{i, 1}^{-1})A_{j, r}^{-1}= T_{i, j}T_{i, j-1}^{-1}$ $\hspace{2.2cm} 1\leq i<j \leq n, \;\; 1\leq r \leq 2g-1$;
\item[{\rm(PR4)}] $T_{i, j}T_{k, l}= T_{k, l}T_{i, j} \hspace{3.9cm} 1\leq i<j<k<l \leq n \;\; \text{or} \;\; 1\leq i<k<l\leq j\leq n$;
\item[{\rm(PR5)}] $T_{k, l}T_{i, j}T_{k, l}^{-1}= T_{i, k-1}T_{i, k}^{-1}T_{i, j}T_{i, l}^{-1}T_{i, k}T_{i, k-1}^{-1}T_{i, l}$ $\hspace{3.6cm} 1\leq i<k \leq j<l \leq n$;
\item[{\rm(PR6)}] $a_{i, r}T_{j, k}= T_{j, k}a_{i, r}  \hspace{3cm} 1\leq i<j<k \leq n \;\; \text{or} \;\; 1\leq j<k<i \leq n, \;\; 1\leq r \leq 2g $;
\item[{\rm(PR7)}] $a_{i, r}(a_{j, 2g}^{-1}\cdots a_{j, 1}^{-1}T_{j, k}a_{j, 2g}\cdots a_{j, 1}) = (a_{j, 2g}^{-1}\cdots a_{j, 1}^{-1}T_{j, k}a_{j, 2g}\cdots a_{j, 1})a_{i, r} \hspace{0.8cm} 1\leq j<i \leq k \leq n$;
\item[{\rm(PR8)}] $T_{j,n}= \left(\displaystyle\prod_{i= 1}^{j-1}{a_{i, 2g}^{-1}\cdots a_{i, 1}^{-1}T_{i, j-1}T_{i, j}^{-1}a_{i, 1}\cdots a_{i, 2g}}\right)a_{j, 1}\cdots a_{j, 2g}a_{j, 1}^{-1}\cdots a_{j, 2g}^{-1}$;}
\end{itemize}
where:\\ 
\begin{align*}
A_{j, s} &= a_{j, 1}\cdots a_{j, s-1}a_{j, s+1}^{-1}\cdots a_{j, 2g}^{-1},\\
T_{i,j}   &=  t_{i,j}\cdots t_{i,i+1}.
\end{align*}
\end{itemize}
\end{teo}

%
%

\vspace{1cm}

\subsection{\bf{A method for finding presentations of groups}}

\begin{propo}{\rm\cite[Proposition $1$, pp.$138$--$140$]{Johnson}}\label{Prop1.3.1}
Consider the following short exact sequence of groups and homomorphisms:
\begin{eqnarray*}\label{sequencia de grupos}\xymatrix{
1  \ar[r]      &      A      \ar[r]^{i}       &      \widetilde{G}    \ar[r]^{p}       &     G   \ar[r]      &       1,\\
}\end{eqnarray*} 
Suppose that the groups $A$ and $G$ admit presentations $\langle X; R_{A}\rangle$ and $\langle Y; R_{G}\rangle$ respectively, where $X$ and $Y$ are sets of generators, while $R_{A}$ and $R_{G}$ are sets of relations. The following well-known procedure outlines a method for putting together a presentation of~$\widetilde{G}$:\\
\begin{itemize}
\item[] {\bf{Generators of $\widetilde{G}$:}} Let $\widetilde{X}= \{\tilde{x}= i(x); \;\; x \in X\}$ be the images of the generators $X$ of $A$ under the homomorphism $i$. Now, given $y \in Y$, let $\tilde{y}$ denote a chosen pre-image of $y$ under $p$, i.e., $p(\tilde{y}) = y$. Define $\widetilde{Y} = \{\tilde{y}; \;\; y \in Y\}$ the set of all such pre-images. Then $\widetilde{X}\cup \widetilde{Y}$ constitute a set of generators for $\widetilde{G}$.\\
\item[] {\bf{Relations:}} There are three types of relations in $\widetilde{G}$:\\
\begin{itemize}
\item[] {\bf{Type $1$:}} Relations of the form $\widetilde{R}_{A} = \{\tilde{r}_{A}; \;\; r_{A} \in R_{A}\}$; where $\widetilde{R}_{A}$ is the set of words in $\widetilde{X}$ obtained from $R_{A}$ by replacing each $x$ by $\tilde{x}$. Thus each $\tilde{r}_{A}$ is an image under  the injective homomorphism $i$ of a relation $r_{A}$ in $\widetilde{G}$;\\
\item[] {\bf{Type $2$:}} Let $\tilde{r}_{G}$ be a word obtained from a relation $r_{G}$ in $R_{G}$ by replacing each $y$ by its chosen pre-image $\tilde{y}$. We see that $p$ maps $\tilde{r}_{G}$ in $\widetilde{G}$ to relation $r_{G}$ in $G$, therefore $\tilde{r}_{G}$ lies in the $\ker(p)$. Since the sequence (\ref{sequencia de grupos}) is exact, we know that $\ker(p)$ equals the image $i(A)$ of $A$ under the homomorphism $i$. Thus $\tilde{r}_{G} = w_{r}$, where $w_{r}$ is a word in $\widetilde{X}$. Accordingly, we have a second set of relations: 
$$\widetilde{R}_{G}=\{\tilde{r}_{G}= w_{r}; \;\; r_{G} \in R_{G}\};$$
\item[] {\bf{Type $3$:}} Choose any $\tilde{y}$ from the set $\widetilde{Y}$ of chosen pre-images of the generator set $Y$ under $p$. The image of $A$ under $i$ is a normal subgroup of $\widetilde{G}$, therefore each conjugate of generator $\tilde{x}$ again belongs to $i(A)$. Thus $\tilde{y}\tilde{x}\tilde{y}^{-1}$ can be written as a word $w_{x}$ over the generators $\widetilde{X}$ of the kernel. We put  $$\widetilde{C}= \{\tilde{y}\tilde{x}\tilde{y}^{-1}= w_{x}; \;\; x \in X, \;\; \tilde{y} \in \widetilde{Y}\}.$$ 
\end{itemize}
\end{itemize}
With the previous notation, the group $\widetilde{G}$ has the presentation $$\langle\widetilde{X},\widetilde{Y}; \;\; \widetilde{R}_{A}, \widetilde{R}_{G}, \widetilde{C}\rangle.$$
\end{propo}

\vspace{4cm}

\section{\bf{Main results: homotopy generalized string links over surfaces}\label{secao3}}

\vspace{0.2cm}

\subsection{\bf{$\bf{\widehat{B}_{n}(M)}$ group structure}}

Let $M$ be a compact, not necessarily orientable surface of genus $g \geq 1$. Choose $n$ points $\mathcal{P}= \{P_{1},\ldots, P_{n}\}$ to lie in the interior of $M$. Let $I_{1},\ldots,I_{n}$ be $n$ copies of the  interval $I=[0,1]$ and $\coprod_{i=1}^{n}{I_{i}}$ denote the disjoint union of these intervals.

\vspace{0.2cm}

\begin{defi}\label{def3.1}
A \emph{generalized string link $\sigma$ on $n$ strands on a surface $M$} is a smooth or piecewise linear proper embedding $\sigma: \coprod_{i= 1}^{n}{I_{i}} \rightarrow M \times I,$ which fulfills the two following conditions:
\begin{itemize}
\item[{\rm(i)}] $\sigma|_{(I_{i}(0))}= (P_{i},0)$,
\item[{\rm(ii)}] $\sigma|_{(I_{i}(1))} \in \{(P_{1},1),\ldots,(P_{n},1)\}$,
\end{itemize}
where $I_{i}(t)=t\in I_{i}$, for all $t$ and for all $i=1,\ldots,n$.
\end{defi}

\begin{figure}[h]
\center
\includegraphics[scale=0.7]{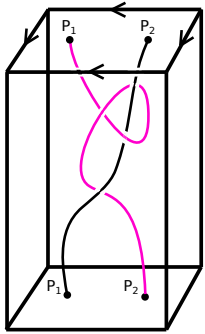}
\caption{A generalized string link $\sigma$ over the $2$-dimensional torus. \label{Fig2.1}}
\end{figure}

Here, we orient the strands downwards from $M \times \{0\}$ to $M \times\{1\}$. Besides, an \emph{ambient isotopy} between generalized string links $\sigma$ and $\sigma^{\prime}$ is an orientation-preserving diffeomorphism of $M \times I$ which maps $\sigma$ onto $\sigma^{\prime}$, keeping the boundary $M \times \{0,1\}$ point-wise fixed and is isotopic to the identity, relative to $M \times \{0,1\}$.
When $\sigma|_{(I_{i}(0))}= (P_{i},1)$, we just obtain a string link, i.e., a string link is a pure generalized string link described in \cite{Yurasovskaya}. When the surface in question is the disk $\mathbb{D}$, we have the homotopy braids described in \cite{Goldsmith}.

Definition $\ref{linkhomotopiaH}$ given in \cite{Habegger} allows us to extend its notion for two {\it{generalized}} string links, since the non-trivial permutation induced by their strands doesn't interfere in the deformation process, which keeps $M \times \{0,1\}$ fixed. Therefore, we have the following definition: 

\begin{defi}\label{linkhomotopia1} 
Two generalized string links $\sigma$ and $\sigma^{\prime}$ are \emph{link-homotopic} if there is a homotopy of the strings in $M \times I$, fixing $M \times \{0,1\}$ and deforming $\sigma$ to $\sigma^{'}$, such that the images of different strings remain disjoint during the deformation. During the course of deformation, each individual strand is allowed to pass through itself but not through other strands. 
\end{defi}

Equivalently, 

\begin{defi}\label{linkhomotopia2}
We say \emph{link-homotopy} is an equivalence relation on generalized string links that is generated by a finite sequence of ambient isotopies of $M \times I$ fixing $M \times \{0,1\}$, and local \emph{crossing changes} of arcs from the \emph{same strand} of a generalized string link called {\it{link-homotopy moves}}. 
\end{defi}

The property of the local crossing changes consists in considering the undercrossing and the overcrossing as the same crossing (in the same strand), as shown in Figure  \ref{Fig2.2}, i.e, a crossing change for generalized string links remains the same as defined for string links.

\vspace{0.3cm}

\begin{figure}[h]
\center
\includegraphics[scale=0.35]{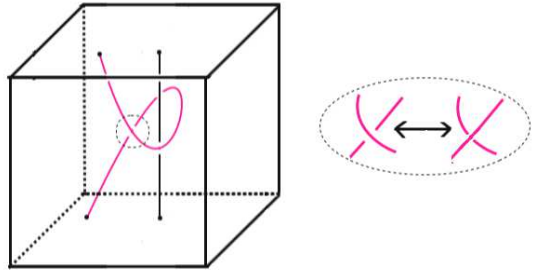}
\caption{A crossing change. \label{Fig2.2}}
\end{figure}

\begin{lema}\label{lema4.2.1}
Let $\sigma$ and $\sigma^{\prime}$ be generalized string links over a surface $M$. Denote by $\sim$ the link-homotopy equivalence relation. If $\sigma \sim \sigma^{\prime}$, then $\sigma \tau \sim \sigma^{\prime} \tau$ and $\tau \sigma \sim \tau \sigma^{\prime}$, for any generalized string links $\tau$.
\end{lema}
\begin{proof} In both cases, consider the concatenation of the mentioned generalized string links in its levels in the diagram of braids respectively. Thus we can deform $\sigma$ to $\sigma^{\prime}$ under homotopy while $\tau$ is fixed, for any generalized string link $\tau$. 
\end{proof}

\begin{teo}\label{Teorema4.2.2}
Every $n$-strand generalized string link on a surface $M$ is link-homotopic to a braid.
\end{teo}
\begin{proof} Let us denote by $\sim$ the link-homotopy equivalence relation. We would like to show that if $\sigma$ is a generalized string link, then $\sigma \sim \alpha$, for some $\alpha \in B_{n}(M)$.
Consider $\sigma$ a generalized string link on $n$ strands and $\beta$ some braid on $n$ strands such that the concatenation $\sigma\beta$ is a string link, namely $\sigma^{\prime}$.
Thus we have $\sigma\beta \sim \sigma^{\prime}$. Since $\sigma^{\prime}$ is a string link, it is link-homotopic to a pure braid on $n$-strands, namely $\gamma$. Consequently, $\sigma^{\prime} \sim \gamma$. By the transitivity of the equivalence relation, we have $\sigma \beta \sim \gamma$.
Let $\beta^{-1}$ be the inverse of the braid $\beta$. According to Lemma \ref{lema4.2.1}, we have $\sigma \beta \beta^{-1} \sim \gamma \beta^{-1}$, i.e., $\sigma \sim \gamma \beta^{-1}$, where $\gamma\beta^{-1}$ is a braid on $n$-strands. Declare $\gamma\beta^{-1}=\alpha$.
Therefore, every generalized string link on $n$-strands is link-homotopic to a braid on $n$ strands.
\end{proof} 

Theorem \ref{Teorema4.2.2} tells us that, up to link-homotopy, we can treat generalized string links as braids.\\

From here forward, we let $M$ denote a closed, connected and orientable surface of genus $g \geq 1$. \\

\begin{propo}\label{Prop4.2.4}
The group of link-homotopically trivial surface braids on n-strands, namely $H_{n}(M)$, is a normal subgroup of $B_{n}(M)$.
\end{propo}
\begin{proof} Clearly, the subgroup $H_{n}(M)$ of $PB_{n}(M)$ is also a subgroup of $B_{n}(M)$ (see \cite{Yurasovskaya}). We show that $\beta H_{n}(M) \beta^{-1} \subseteq H_{n}(M)$, for all $\beta$ in $B_{n}(M)$. Indeed, given $\beta \in B_{n}(M)$, $\sigma \in H_{n}(M)$, consider the braid diagram of $\beta \sigma \beta^{-1}$. Since $\sigma $ is link-homotopic to the trivial braid, namely $1$, and keeping $\beta$ and $\beta^{-1}$ fixed while deforming $\sigma$ to the trivial braid, we have:

\begin{figure}[h]
\center
\includegraphics[scale=0.6]{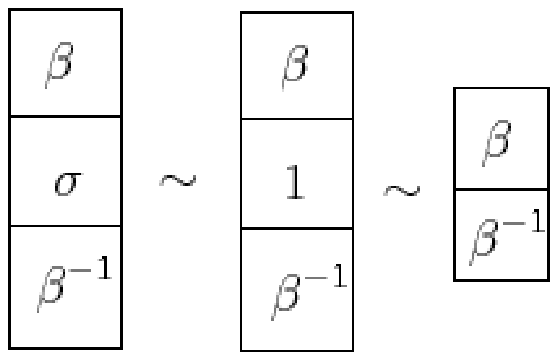}
\end{figure}


Therefore, $\beta H_{n}(M) \beta^{-1}\subseteq H_{n}(M)$ and $H_{n}(M)$ is a normal subgroup of $B_{n}(M)$ as required.
\end{proof}

\vspace{0.2cm}

We denote the set of link-homotopy classes of generalized string links on a surface $M$ by $\widehat{B}_{n}(M)$, which we shall call simply the group of homotopy generalized string links on $M$. $\widehat{B}_{n}(M)$ inherits the operation (concatenation) and the inverse (mirror reflexion up to link-homotopy) from $B_{n}(M)$.

\begin{teo}\label{Teorema4.2.6}
Under concatenation, $\widehat{B}_{n}(M)$ is isomorphic to the quotient of the braid group $B_{n}(M)$ by the normal subgroup of link-homotopically trivial braids $H_{n}(M)$: 
\begin{eqnarray*}{\widehat{B}_{n}(M)} \simeq \frac{B_{n}(M)}{H_{n}(M)}.\end{eqnarray*}
\end{teo}
\begin{proof} Let us consider the natural map: 
\begin{eqnarray*} f: B_{n}(M) \rightarrow \widehat{B}_{n}(M),
\end{eqnarray*}
defined by $f(\beta)= \hat{\beta}$, which takes the isotopy class of each braid to its link-homotopy class. Clearly, there are braids which are not trivial up to isotopy but trivial up to link-homotopy. They form the set of generators of $H_{n}(M)$. Denote by $\approx$ the isotopy equivalence relation and by $\sim$ the link-homotopy equivalence relation.

Firstly, we show that $f$ is a well defined homomorphism: indeed, let $\beta$ and $\gamma$ be two representative braids of the same equivalence class. Consequentely, we have $\beta \approx \gamma$, $f(\beta)=\hat{\beta}$ and $f(\gamma)=\hat{\gamma}$, where $\hat{\beta}$ and $\hat{\gamma}$ are string links provided from $\beta$ and $\gamma$ under finite link-homotopy moves, respectively. Thus, $f(\beta)=\hat{\beta} \sim \beta$, $f(\gamma)=\hat{\gamma} \sim \gamma$ and $f(\gamma)=\hat{\gamma} \sim \gamma \approx \beta \sim \hat{\beta}= f(\beta)$. Since isotopy implies link-homotopy, we have $f(\beta)= f(\gamma)$. Accordingly, $f$ is well defined.

By Theorem \ref{Teorema4.2.2}, $f$ is surjective. Thus, $\dfrac{B_{n}(M)}{\ker(f)}\simeq \widehat{B}_{n}(M)$.

Finally, we claim that $\ker(f)=H_{n}(M)$. Indeed, $\ker(f)= \{\beta \in B_{n}(M); f(\beta)=1\}$. If $\beta \in \ker(f)$, then we have that $\beta$ is a braid link-homotopic to the trivial braid, provided by a finite sequence of link-homotopic moves. Moreover, $\beta$ must be pure since permutation is a braid invariant. Therefore, we find that $\ker(f)\subseteq H_{n}(M)$. Conversely, if $\beta \in H_{n}(M)$, then $\beta$ is link-homotopic to the trivial braid. Choose a generalized string link $\hat{\beta}$ that is link-homotopic to $\beta$ under a finite sequence of link-homotopic moves. Then, $\beta \in \ker(f)$.
Therefore, we have $\widehat{B}_{n}(M) \simeq \dfrac{B_{n}(M)}{H_{n}(M)}$, as required.
\end{proof}

\vspace{0.5cm}

Now, given a generalized string link $\sigma$, let us denote by $\pi(\sigma)$ the permutation induced by $\sigma$. Let $\Sigma_{n}$ be the symmetric group on $n$ elements, and consider the map $\psi: \widehat{B}_{n}(M) \rightarrow \Sigma_{n}$ defined by $\psi(\sigma)= \pi(\sigma)$, for all $\sigma$ in $\widehat{B}_{n}(M)$. 

We claim that $\psi$ is a well defined homomorphism. Indeed, if $\sigma$ and $\sigma^{\prime}$ are two generalized string links in the same equivalence class, then both have the same permutation. Thus, $\psi(\sigma)= \psi(\sigma^{\prime})$ and the map is well defined as claimed. Clearly, $\psi$ is a homomorphism and is also surjective by construction. Thus,  $\dfrac{\widehat{B}_{n}(M)}{\ker(\psi)}$ is isomorphic to $\Sigma_{n}$. By definition, $\ker(\psi)= \widehat{PB}_{n}(M)$. Consequently, we have the following result:

\begin{propo}\label{Prop4.2.7}
$\widehat{PB}_{n}(M)$ is a normal subgroup of $\widehat{B}_{n}(M)$. Moreover, under the homomorphism $\psi$ defined previously, we have the well defined short exact sequence:
$$\xymatrix{
                      1 \ar[r]   &   \widehat{PB}_{n}(M) \ar[r]^{i}   &   \widehat{B}_{n}(M) \ar[r]^{\hspace{0.2cm}\psi}   &   \Sigma_{n} \ar[r]   &   1
},$$ where $i$ is the inclusion homomorphism.
\end{propo} 

\subsection{A presentation for the group of homotopy generalized string links on surfaces}

Since we have defined the group homotopy generalized string links over an orientable surface $M$ of genus $g\geq 1$, which is, informally speaking, a generalization for the braid group for surfaces, a classical topic appears: the search for its presentation. Furthermore, we would like to find if $\widehat{B}_{n}(M)$ is finitely presented, i.e, if there is a finite set of generators and relations which define this group. The purpose of this section is to prove the following result:

\vspace{0.2cm}

\begin{teo}\label{Teorema4.4.1}
Let $M$ be a closed, orientable surface of genus $g\geq 1$. The group of link-homotopy classes of generalized string links over $M$, namely $\widehat{B}_{n}(M)$, admits the following presentation:
\begin{itemize}
\item[] {\bf{Generators:}} $\{a_{1,1},\ldots,a_{1,2g}\} \cup \{\sigma_{1},\ldots,\sigma_{n-1}\}$;
\item[] {\bf{Relations:}}
\begin{itemize}
\item[{\rm(LH)}] $[{t_{1,j}},{t^{h}_{1,j}}]=1 \hspace{8.5cm} h \in \mathbb{F}(2g+n-1)$;
\item[{\rm(R1)}] $\sigma_{i}\sigma_{j}= \sigma_{j}\sigma_{i}  \hspace{9.9cm} |i-j|\geq 2$;
\item[{\rm(R2)}] $\sigma_{i}\sigma_{i+1}\sigma_{i} = \sigma_{i+1}\sigma_{i}\sigma_{i+1}  \hspace{7.6cm} 1\leq i \leq n-2$;
\item[{\rm(R3)}]
$\mathit{a}_{1, 1}\cdots\mathit{a}_{1, 2g}\mathit{a}_{1, 1}^{-1}\cdots\mathit{a}_{1, 2g}^{-1}=
\sigma_{1}\cdots\sigma_{n-2}\sigma_{n-1}^{2}\sigma_{n-2}\cdots\sigma_{1}$
\item[{\rm(R4)}] $\mathit{a}_{1, r}A_{2, s}= A_{2, s}\mathit{a}_{1, r} \hspace{4.9cm} 1\leq r \leq 2g \ 1\leq s \leq 2g-1, \ r\neq s$;
\item[{\rm(R5)}] $(\mathit{a}_{1,1}\cdots\mathit{a}_{1, r})A_{2, r}=
\sigma_{1}^{2}A_{2, r}(\mathit{a}_{1, 1}\cdots\mathit{a}_{1, r}) \hspace{4.7cm} 1\leq r \leq 2g-1$;
\item[{\rm(R6)}] $\mathit{a}_{1, r}\sigma_{i}= \sigma_{i}\mathit{a}_{1, r} \hspace{8.3cm} 1\leq r \leq 2g; \  i\geq 2$;
\end{itemize}
\end{itemize}where:
\begin{eqnarray*} 
t_{1,j}&=& \sigma_{1}\cdots\sigma_{j-2}\sigma_{j-1}^{2}\sigma_{j-2}^{-1} \cdots \sigma_{1}^{-1}, \hspace{0.3cm} j=2,\ldots,n, \\ 
A_{2,s}&=&\sigma_{1}^{-1}(a_{1,1}\cdots a_{1,s-1}a_{1,s+1}^{-1}\cdots a_{1,2g}^{-1})\sigma_{1}^{-1},  \hspace{0.3cm} s=1,\ldots,2g-1.
\end{eqnarray*}
\end{teo}

\begin{figure}[h]
\center
\includegraphics[scale=0.5]{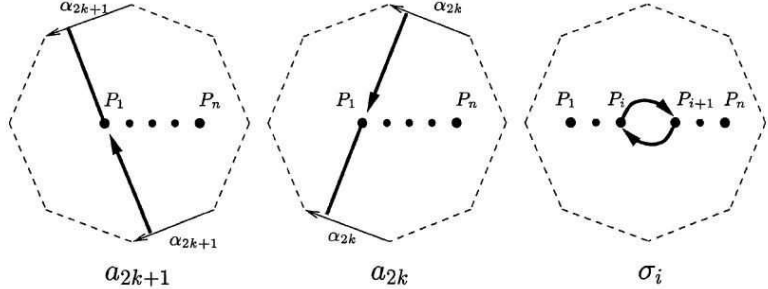}
\vspace{-11cm}
\caption{Generators of $\widehat{B}_{n}(M)$, where $a_{r}=a_{1,r}$ \cite{Gonzales}.}
\end{figure}

To prove Theorem \ref{Teorema4.4.1} we apply methods which are similar to those used by González--Meneses in \cite{Gonzales} for computing the presentation of the surface braid group $B_{n}(M)$ over a closed surface $M$ discussed in Proposition \ref{Prop1.3.1}. Such methods are well known and classical in the literature for computing presentations of groups. In this paper, we use the notations and some arguments by \cite{Gonzales} to establish connections with the presentations of the braid groups $B_{n}(M)$ and the generalized string links groups  $\widehat{B}_{n}(M)$ guaranteed by Theorem \ref{Teorema4.2.6}. \\

\noindent{\bf{The idea of the proof:}}\\

Recall the short exact sequence from Proposition \ref{Prop4.2.7}:
$$\xymatrix{
                      1 \ar[r]   &   \widehat{PB}_{n}(M) \ar[r]^{i}   &   \widehat{B}_{n}(M) \ar[r]^{\hspace{0.2cm}\psi}   &   \Sigma_{n} \ar[r]   &   1  \hspace{1cm} (\star)
}$$ 

Now, consider the presentation of $\widehat{PB}_{n}(M)$ given in Theorem \ref{Teorema2.1.5} and the presentation for the symmetric group on $n$ elements, namely $\Sigma_{n}$ \cite{Harpe}, as follows: \\

\textbf{Presentation of $\Sigma_{n}$:}
\begin{itemize}
\item[]{\textbf{Generators:}} $\delta_{1},\ldots, \delta_{n-1}$.
\item[]{\textbf{Relations:}}
\begin{itemize}
\item[(SR1)] $\delta_{i}\delta_{j}= \delta_{j}\delta_{i}$ \hspace{3.3cm} $|i-j|\geq 2$;
\item[(SR2)] $\delta_{i}\delta_{i+1}\delta_{i}= \delta_{i+1}\delta_{i}\delta_{i+1}$ \hspace{1.1cm} $1\leq i \leq n-2$;
\item[(SR3)] $\delta_{i}^{2}= 1$ \hspace{3.4cm} $1\leq i \leq n-1$;
\end{itemize}
where $\delta_{i}$ is the permutation $(i, \ i+1)$, for all $i$.
\end{itemize}

\vspace{0.3cm}

We suppose that there is an abstract group, namely $\mathcal{B}_{n}$, given by the presentation of Theorem $\ref{Teorema4.4.1}$. We then define a map $\varphi: \mathcal{B}_{n} \rightarrow \widehat{B}_{n}(M)$ in the natural way, and we prove that it is a well defined homomorphism. Finally, we apply Proposition $\ref{Prop1.3.1}$ to the exact sequence $(\star)$ given in Proposition $\ref{Prop4.2.7}$ to show that $\varphi$ is an isomorphism.

Observe that the presentation for $\widehat{B}_{n}(M)$ given in Theorem $\ref{Teorema4.4.1}$ differs from the Gonzalez-Meneses presentation \cite[Theorem $2.1$]{Gonzales} by the link-homotopic relation ${\rm{(LH)}}$. This is a relation of Type $1$ according to Proposition \ref{Prop1.3.1}. Most computations involving the generators and relations from \cite{Gonzales} will not be repeated. During the proof, we give one of these relations as an example and it remains unaltered from \cite{Gonzales}. Moreover, all the details omitted during the proof can be found in \cite{Gonzales}.\\

\begin{proof}[Proof of Theorem  \ref{Teorema4.4.1}]
Let us call $\mathcal{B}_{n}$ the abstract group that admits the presentation of Theorem \ref{Teorema4.4.1}. To show the validity of the presentation, we need to add some auxiliary generators and relations:

\begin{itemize}
\item[] {\bf{new generators:}}
\begin{itemize}
\item[-] $a_{i,r}$,\;\; $2\leq i \leq n$;
\item[-] $t_{j,k}$,\;\; $1 \leq j< k \leq n$.
\end{itemize}
\item[] {\bf{new relations:}}
\begin{itemize}
\item[(R7)] $a_{j+1,r}=\sigma_{j}a_{j,r}\sigma_{j} \hspace{4.8cm}1\leq j \leq n-1; 1\leq r \leq 2g$; $r$ even;
\item[(R8)] $a_{j+1,r}=\sigma_{j}^{-1}a_{j,r}\sigma_{j}^{-1} \hspace{4.4cm}1\leq j \leq n-1; 1\leq r \leq 2g$; $r$ odd;
\item[(R9)] $t_{i,j}= \sigma_{i}\sigma_{i+1}\cdots\sigma_{j-2}\sigma_{j-1}^{2}\sigma_{j-2}^{-1}\cdots\sigma_{i+1}^{-1}\sigma_{i}^{-1} \hspace{3.8cm}1\leq j<k\leq n$.
\end{itemize}
\end{itemize}

Clearly these new relations still define the same group, i.e., $\mathcal{B}_{n}$. 

It is easy to see that ${\rm{(R7)}}$, ${\rm{(R8)}}$ and ${\rm{(R9)}}$ hold in $\widehat{B}_{n}(M)$, since they hold in $B_{n}(M)$, meaning isotopy implies link-homotopy (see definition \ref{linkhomotopia2}). 

Now let us define the mapping in a natural way:
\begin{eqnarray*}
\begin{split}
\varphi: & \mathcal{B}_{n} \rightarrow \widehat{B}_{n}(M)\\
     & \sigma_{i} \longmapsto \sigma_{i}, \;\; 1\leq i \leq n-1\\
     & a_{1, r} \longmapsto a_{1, r}, \;\; 1\leq r \leq 2g.
\end{split}
\end{eqnarray*}

Observe that we keep the notation $\sigma_{i}$ and $a_{1,r}$ for the braids that will be the images of the generators $\sigma_{i}$ and $a_{1,r}$ of $\mathcal{{B}}_{n}$ under the homomorphism $\varphi$, to avoid excessive notation. Such braids are defined as follows: $\sigma_{i}$ are the elementary braids on $n$ strands on the disk ($B_{n} \hookrightarrow {B}_{n}(M)$) and $a_{1,r}$ are the braids that ``go through the wall of the cylinder$"$ starting and arriving at $P_{1}$, with the remaining strands being trivial. 

We claim that $\varphi$ is well defined, proving that the relations of the presentation of $\mathcal{B}_{n}$ hold in $\widehat{B}_{n}(M)$.
Indeed, according to \cite{Gonzales}, the relations from ${\rm{(R1)}}$ to ${\rm{(R9)}}$ hold in $B_{n}(M)$ since isotopy equivalence relation implies link-homotopy equivalence relation (see Definition \ref{linkhomotopia2}). Thus we find that they still hold in $\widehat{B}_{n}(M)$.

Let us give the braid diagram showing that the relation ${\rm{(R6)}} a_{1,r}\sigma_{i}= \sigma_{i}a_{1,r}$, $1 \leq r \leq 2g$ holds in $\widehat{B}_{n}(M)$, given in \cite{Gonzales}:

\vspace{-0.3cm}

\begin{figure}[h]
\center
\includegraphics[scale=0.2]{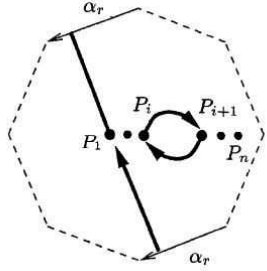}
\vspace{-2.5cm}
\caption{The braid $a_{1,r}\sigma_{i}= \sigma_{i}a_{1,r}$ \cite{Gonzales}.\label{disney4}}
\end{figure}

Now, consider the relation (under $\varphi$) ${\rm{(LH)}} [t_{1,j},t_{1,j}^{h}]=1$, $h \in \mathbb{F}(2g+n-1)$ and recall that the generator set of $\mathbb{F}(2g+n-1)$ is given by $\{a_{1,r}, 1\leq r \leq 2g\} \cup \{t_{1,j}, 2\leq j\leq n\}$. Such relations hold in $\widehat{B}_{n}(M)$, since ${\rm{(LH)}}$ is a particular case from ${\rm{(LH1)}}$ in the presentation of $\widehat{PB}_{n}(M)$, which is contained in $\widehat{B}_{n}(M)$ under the inclusion. 

In order to show that $\varphi$ is surjective, consider the short exact sequence $(\star)$. By applying Proposition \ref{Prop1.3.1} to find the generators of $\widehat{B}_{n}(M)$, we have two types of generators: firstly, the generators of $\widehat{PB}_{n}(M)$ that become generators of $\widehat{B}_{n}(M)$ under the inclusion: $$\{a_{i, r}; \ 1\leq i \leq n, \  1\leq r \leq 2g\}\cup \{t_{j, k}; \  1\leq j < k \leq n\}.$$

Secondly, each $\sigma_{i}$ is a pre-image of $\delta_{i} \in \Sigma_{n}$, for $i=1,\ldots,n-1$, under the surjective homomorphism $\psi$ in the exact sequence $(\star)$. Consequently, we have the generators $$\{\sigma_{i}; 1\leq i \leq n-1\}.$$

Thus, we have $\{a_{i, r}; \ 1\leq i \leq n, \  1\leq r \leq 2g\}\cup \{t_{j, k}; \  1\leq j < k \leq n\} \cup \{\sigma_{i}; 1\leq i \leq n-1\}$ as the full set of generators of $\widehat{B}_{n}(M)$. However, we should note that, by the relation ${\rm{(R9)}}$, $t_{j,k}$ is written as a product of $\sigma_{i}$'s, and by the relations ${\rm{(R7)}}$ and ${\rm{(R8)}}$, the braids $a_{i, r}$ are written as a product of $a_{1, r}$. Accordingly, we reduce the set of generators to $$\{a_{1, r}; \;\; 1\leq r \leq 2g\} \cup \{\sigma_{i}; 1\leq i \leq n-1\}.$$
Therefore, we have $\varphi$ surjective, as required.

Finally, in order to show that $\varphi$ is injective, we need to demonstrate that each relation in $\widehat{B}_{n}(M)$ still holds in $\mathcal{B}_{n}$, i.e., each relation of Type $1,2$ and $3$ according to the Proposition $\ref{Prop1.3.1}$, holds in 
$\mathcal{B}_{n}$. We observe that relations from ${\rm{(R1)}}$ to ${\rm{(R9)}}$ in $\widehat{B}_{n}(M)$ come from the same relations of $B_{n}(M)$, since isotopy implies link-homotopy, and again they are the same for $\widehat{B}_{n}(M)$ and their computations can be found in \cite{Gonzales}.

For the link-homotopy relation ${\rm{(LH)}} \ [t_{1,j},t_{1,j}^{h}]= 1, h \in \mathbb{F}(2g+n-1)$, we prove that the relation ${\rm{(LH1)}}$ given in Theorem \ref{Teorema2.1.5} is a relation of $\widehat{B}_{n}(M)$ under inclusion, namely, ${\rm{(LH)}}$. Consequently, it also holds in $\mathcal{B}_{n}$. Indeed, we start with ${\rm{(R9)}}$:

\begin{eqnarray*}
t_{i,j}  =  \sigma_{i}\sigma_{i+1}\cdots\sigma_{j-2}\sigma_{j-1}^{2}\sigma_{j-2}^{-1}\cdots\sigma_{i+1}^{-1}\sigma_{i}^{-1}\\ 
\Leftrightarrow (\sigma_{1}\sigma_{2}\cdots\sigma_{i-1})t_{i,j}(\sigma_{1}\sigma_{2}\cdots\sigma_{i-1})^{-1} = t_{1,j},
\end{eqnarray*} so, we have
\begin{eqnarray}\label{disney}
t_{i,j} = (\sigma_{1}\sigma_{2}\cdots\sigma_{i-1})^{-1}t_{1,j}(\sigma_{1}\sigma_{2}\cdots\sigma_{i-1}).
\end{eqnarray}

In order to simplify the notation, declare $\alpha =  (\sigma_{1}\sigma_{2}\cdots\sigma_{i-1})$. From Equation $(\ref{disney})$, we have

\begin{eqnarray*}
[t_{1,j},t_{1,j}^{h}] & = & t_{1,j}ht_{1,j}h^{-1}t_{1,j}^{-1}ht_{1,j}^{-1}h^{-1} \\
&=&\alpha(\alpha^{-1}t_{1,j}\alpha) \alpha^{-1}h\alpha(\alpha^{-1}t_{1,j}\alpha)\alpha^{-1}h^{-1}\alpha(\alpha^{-1}t_{1,j}^{-1}\alpha)\alpha^{-1}h\alpha(\alpha^{-1}t_{1,j}^{-1}\alpha)\alpha^{-1}h^{-1}\alpha \alpha^{-1},\\
\end{eqnarray*} and therefore,
\begin{flalign}\label{disney2}
[t_{1,j},t_{1,j}^{h}] = \alpha(t_{i,j}) \alpha^{-1}h\alpha(t_{i,j})\alpha^{-1}h^{-1}\alpha(t_{i,j}^{-1})\alpha^{-1}h\alpha(t_{i,j}^{-1})\alpha^{-1}h^{-1}\alpha \alpha^{-1}.
\end{flalign}

Now, from Equation $(\ref{disney2})$, we have that ${\rm{(LH)}}$ turns out to be equivalent to the relation ${\rm(LH1)}$ from $\widehat{PB}_{n}(M)$:

\begin{eqnarray*}
1&=& [t_{1,j},t_{1,j}^{h}]\\
\Leftrightarrow 1&=& \alpha(t_{i,j}) \alpha^{-1}h\alpha(t_{i,j})\alpha^{-1}h^{-1}\alpha(t_{i,j}^{-1})\alpha^{-1}h\alpha(t_{i,j}^{-1})\alpha^{-1}h^{-1}\alpha \alpha^{-1}\\
\Leftrightarrow 1&=& t_{i,j}{\underbrace{\alpha^{-1}h\alpha}_{g}} t_{i,j}{\underbrace{\alpha^{-1}h^{-1}\alpha}_{g^{-1}}} t_{i,j}^{-1}{\underbrace{\alpha^{-1}h\alpha}} t_{i,j}^{-1}{\underbrace{\alpha^{-1}h^{-1}\alpha}}\\
\Leftrightarrow 1&=& t_{i,j}gt_{i,j}g^{-1}t_{i,j}^{-1}gt_{i,j}^{-1}g^{-1}\\
\Leftrightarrow 1&=& [t_{i,j},t_{i,j}^{g}].
\end{eqnarray*}
With simple calculations, by using ${\rm{(R7)}}$ and ${\rm{(R8)}}$ if necessary, we get that $g$ is an element on the generators of $\mathbb{F}(2g+n-i)$, since $h$ is an element on the generators of $\mathbb{F}(2g+n-1)$. Thus, $\varphi$ is injective. 

Therefore, $\mathcal{B}_{n}$ is isomorphic to $\widehat{B}_{n}(M)$.
\end{proof}

\begin{coro}\label{coro3.10}
Let $S$ be a surface obtained by deleting a single point from a compact, connected and orientable surface without boundary, different from the sphere $\mathbb{S}^{2}$. The group of link-homotopy classes of generalized string links over $S$, namely $\widehat{B}_{n}(S)$, admits the following presentation:
\begin{itemize}
\item[] {\bf{Generators:}} $\{a_{1,1},\ldots,a_{1,2g}\} \cup \{\sigma_{1},\ldots,\sigma_{n-1}\}$;
\item[] {\bf{Relations:}}
\begin{itemize}
\item[{\rm(LH)}] $[{t_{1,j}},{t^{h}_{1,j}}]=1 \hspace{8.4cm} h \in \mathbb{F}(2g+n-1)$;
\item[{\rm(R1)}] $\sigma_{i}\sigma_{j}= \sigma_{j}\sigma_{i}  \hspace{9.9cm} |i-j|\geq 2$;
\item[{\rm(R2)}] $\sigma_{i}\sigma_{i+1}\sigma_{i} = \sigma_{i+1}\sigma_{i}\sigma_{i+1}  \hspace{7.6cm} 1\leq i \leq n-2$;
\item[{\rm(R4)}] $\mathit{a}_{1, r}A_{2, s}= A_{2, s}\mathit{a}_{1, r} \hspace{4.9cm} 1\leq r \leq 2g \ 1\leq s \leq 2g-1, \ r\neq s$;
\item[{\rm(R5)}] $(\mathit{a}_{1,1}\cdots\mathit{a}_{1, r})A_{2, r}=
\sigma_{1}^{2}A_{2, r}(\mathit{a}_{1, 1}\cdots\mathit{a}_{1, r}) \hspace{4.7cm} 1\leq r \leq 2g-1$;
\item[{\rm(R6)}] $\mathit{a}_{1, r}\sigma_{i}= \sigma_{i}\mathit{a}_{1, r} \hspace{8.3cm} 1\leq r \leq 2g; \  i\geq 2$;
\end{itemize}
\end{itemize}where:
\begin{eqnarray*} 
t_{1,j}&=& \sigma_{1}\cdots\sigma_{j-2}\sigma_{j-1}^{2}\sigma_{j-2}^{-1} \cdots \sigma_{1}^{-1}, \hspace{0.3cm} j=2,\ldots,n, \\ 
A_{2,s}&=&\sigma_{1}^{-1}(a_{1,1}\cdots a_{1,s-1}a_{1,s+1}^{-1}\cdots a_{1,2g}^{-1})\sigma_{1}^{-1},  \hspace{0.3cm} s=1,\ldots,2g-1.
\end{eqnarray*}
\end{coro}
\begin{proof} 
The proof follows the steps of Theorem \ref{Teorema4.4.1}, using the presentation of Theorem \ref{Teorema2.6}.
\end{proof}

\vspace{0.5cm}

With a view to understanding how this group structure fits into the literature, consider the following exact sequence that appears as a consequence of Theorem \ref{Teorema4.2.6}:
$$\xymatrix{
                      1 \ar[r]   &   H_{n}(M) \ar[r]^{i}   &   {B}_{n}(M) \ar[r]^{\hspace{0.2cm}\varphi}   &    \ar[r]   \widehat{B}_{n}(M) &   1 \hspace{1.0cm}  (\star \star)
},$$ where $i$ is the inclusion and $\widehat{B}_{n}(M)$ is the natural projection given in the proof of Theorem \ref{Teorema4.2.6}.

According to \cite[Lemma $3.1$, p. 269]{Rolfsen}, if $H_{n}(M)$ and $ \widehat{B}_{n}(M)$ are left-orderable groups, then so is ${B}_{n}(M)$. $H_{n}(M)$ is left-orderable since it is contained in the bi-orderable group $PB_{n}(M)$  \cite{Gonzales}, hence the search for the left-orderability of $\widehat{B}_{n}(M)$ becomes an attractive question to pursue.
Besides this question, let us consider other ones:
\begin{itemize}
\item[-] about the torsion-freeness of $\widehat{{B}}_{n}(M)$, for all $n \geq 2$. When $M$ is the disk, the torsion-freeness of $\widehat{{B}}_{n}$ is still an open question, for all $n \geq 6$ (for the torsion-freeness of $\widehat{{B}}_{n}$, for $n= 2,3,4$ and $5$, see \cite{Humphries}).
\item[-] about a possible faithful representation theorem for $\widehat{{B}}_{n}(M)$ and $\widehat{{B}}_{n}$ along the lines of Artin's representation for braids on the disk in \cite{Artin} and for braids on surfaces in \cite{Bellingeri2};
\item[-] about a possible solution for the word problem of $\widehat{{B}}_{n}(M)$ and $\widehat{{B}}_{n}$. 
\end{itemize}

\vspace{0.5cm}

\noindent{\bf{Acknowledgment:}} I would like to thank Professor Dale Rolfsen, who encouraged me to write this paper during the period he was my Phd co-advisor at UBC, Vancouver, Canada.

\small
\printindex


\normalsize

\begin{thebibliography}{99}

\bibitem[A]{Artin} Artin, E. Theory of braids. Annals of Mathematics $(1947)$, $101-126$.
\bibitem[B]{Bellingeri2} Bardakov, V. G. and Bellingeri, P. On representations of Artin--Tits and surface braid groups. Journal of Group Theory $14$, $1(2011)$, $143-163$.
\bibitem[D]{Dehornoy} Dehornoy, P. Braid groups and left distributive operations. Trans. Amer. Math. Soc. $345$, $1(1994)$, $115-150$.
\bibitem[DDRW]{Rolfsen} Dehornoy, P., Dynnikov, i., Rolfsen, D. and Wiest, B. Ordering braids. $no. 148$. American Mathematical Soc., $2008$.
\bibitem[G]{Goldsmith} Goldsmith, D. L. Homotopy of braids: in answer to a question of E. Artin. In Topology Conference $(1974)$, Springer, pp. $91-96$.
\bibitem[GM]{Gonzales} González-Meneses, J. New presentations of surface braid groups. Journal of Knot Theory and Its Ramifications $10$, $03(2001)$, $431-451$.
\bibitem[GM2]{Gonzales2} González-Meneses, J. Ordering pure braid groups on compact, connected surfaces. Pacific Journal of Mathematics $203$, $2(2002)$, $369-378$.
\bibitem[H]{Habegger} Habegger, N. and Lin, X.-S. The classification of links up to link-homotopy. Journal of the American Mathematical Society $3$, $2(1990)$, $389-419$.
\bibitem[Hu]{Humphries}, S. P. Torsion-free quotients of braid groups. International Journal of Algebra and Computation $11$, $3(2001)$, $363-373$.
\bibitem[J]{Johnson} Johnson, D. L. Presentation of Groups.$no. 15$. Cambridge University Press, $1997$.
\bibitem[Le]{Levine} Levine, J. P. An approach to homotopy classification of links. Transactions of the American Mathematical Society $306$, $1(1988)$, $361-387$.
\bibitem[LH]{Harpe} La Harpe, P. d. An invitation to coxeter groups. In Group theory from a geometrical viewpoint, $1991$.
\bibitem[LM]{Lima} Theodoro de Lima, J. R. and de Mattos, D. Ordering homotopy string links over surfaces. Journal of Knot Theory and Its Ramifications $25$, $1(2016)$, 165001. 
\bibitem[Mil]{Milnor} Milnor, J. Link Groups. Annals of Mathematics $(1954)$, $177-195$.
\bibitem[RZ]{Zhu} Rolfsen, D. and Zhu, J. Braids, orderings and zero divisors. Journal of Knot Theory and Its Ramifications $7$, $6(1998)$, $837-841$.
\bibitem[Y]{Yurasovskaya} Yurasovskaya, E. Homotopy string links over surfaces. PhD thesis, University of British Columbia, $2008$.

\end{thebibliography}
\end{document}